\documentclass{amsart}

\usepackage{amsmath, amsthm, amscd, amssymb}

\newcommand{\R}{\mathbf R}

\newcommand{\Q}{\mathbf Q}
\newcommand{\C}{\mathbf C}

\newcommand{\F}{\mathcal F}

\renewcommand{\phi}{\varphi}

\newcommand{\eps}{\epsilon}

\newcommand{\bs}{\backslash}

\newcommand{\bmx}{\left( \begin{matrix}}
\newcommand{\emx}{\end{matrix} \right)}

\newcommand{\calH}{\mathcal H}

 \DeclareMathOperator{\tr}{tr} 
  \DeclareMathOperator{\Sp}{Sp}

  \DeclareMathOperator{\Nm}{Nm}
    \DeclareMathOperator{\Aut}{Aut}

\newtheorem{theorem}{Theorem}[section]
\newtheorem{lemma}[theorem]{Lemma}
\newtheorem{corollary}[theorem]{Corollary}

\newtheorem{proposition}[theorem]{Proposition}

\begin{document}

\title{Exact averages of central values of triple product $L$-functions}

\author{Brooke Feigon}
\address{Department of Mathematics\\University of Toronto\\Toronto, ON\\Canada M5S 2E4}
\email{bfeigon@math.toronto.edu}

\author{David Whitehouse}
\address{Department of Mathematics\\MIT, 2-171\\77 Massachusetts Avenue\\Cambridge, MA 02139-4307}
\email{dw@math.mit.edu}

\begin{abstract}
We obtain exact formulas for central values of triple product $L$-functions averaged over newforms of weight $2$ and prime level. We apply these formulas to non-vanishing problems. This paper uses a period formula for the triple product $L$-function proved by Gross and Kudla.
 \end{abstract}

\maketitle

\tableofcontents

\section{Introduction}

Let $f$, $g$ and $h$ be normalized holomorphic modular forms which are eigenfunctions for the Hecke operators. Associated to such a triple is the triple product $L$-function
\[
L(s, f\otimes g\otimes h) = \prod_p L_p(s, f\otimes g\otimes h)
\]
defined by an Euler product of degree 8 which converges for $\Re s \gg 0$ (see Section \ref{sec:modforms} for the definition of the local factors). An integral representation for $L(s, f\otimes g\otimes h)$ was first obtained by Garrett \cite{garrett:1987} using an Eisenstein series on $\Sp(6)$. Garrett treated the case that $f$, $g$ and $h$ are all of full level and have the same weight. His method was generalized further by Piatetski-Shapiro and Rallis \cite{piatetski-shapiro:1987} using an adelic approach. The integral representation yields the meromorphic continuation of $L(s, f\otimes g\otimes h)$ to the complex plane as well as a functional equation.

Naturally the central value of $L(s, f\otimes g\otimes h)$ is of considerable interest. In this paper we consider the case when $f$, $g$ and $h$ are of weight two and of the same prime level $N$. Let $\F_2(N)$ denote the set of such forms. We obtain (see Section \ref{section:avgs}) exact formulas for the average of $L(2, f\otimes g\otimes h)$ (the central value of $L(s, f\otimes g \otimes h)$) weighted by Hecke eigenvalues as one varies across the set $\F_2(N)$ while keeping none, one or two of the forms fixed. One of the main results is given by,

\begin{theorem}
Let $N$ be prime. Then for any $h \in \F_2(N)$,
\begin{align*}
4 \pi N \sum_{f, g \in \F_2(N)} \frac{L(2, f\otimes g\otimes h)}{(f, f)(g, g)}
\end{align*}
equals
\begin{align*} \left(1-\frac{24}{N-1}\right) (h,h)+
\begin{cases}
0; & N \equiv 1 \bmod 12\\
6\sqrt 3{L(1, h) L(1, h\otimes \chi_{-3})};
& N \equiv 5 \bmod 12\\
4{L(1, h) L(1, h\otimes \chi_{-4})}; & N \equiv 7 \bmod 12\\
6\sqrt 3{L(1, h) L(1, h \otimes \chi_{-3})}+4{L(1, h) L(1, h \otimes \chi_{-4})}; & N \equiv 11 \bmod 12.
\end{cases}
\end{align*}
\end{theorem}

In this theorem $( \ , \ )$ denotes the Petersson inner product normalized as in Section \ref{sec:modforms} below. We note an interesting feature of this result is the appearance of central values of smaller $L$-functions on the right hand side.

In Section \ref{sec:conseq} we obtain some consequences of this formula on the non-vanishing of $L(2, f\otimes g\otimes h)$. For example two corollaries of this theorem are,

\begin{corollary}
Let $N>25$ be prime. For $h \in \F_2(N)$,
\[
\#\{ (f,g)\in \F_2(N)\times \F_2(N): L(2, f\otimes g \otimes h)\neq 0\} \gg_\epsilon N^{3/4-\epsilon}.
\]
\end{corollary}

\begin{corollary}
Let $p$ be prime and let $\mathcal P$ be a place in $\overline \Q$ above $p$. Let $N$ be prime such that $N \equiv 1 \bmod 12$ and $p\nmid(N-25)$. Then for any $h \in \F_2(N)$, there exist $f, g \in \F_2(N)$ such that
\[
L^{alg}(2, f\otimes g \otimes h) \not\equiv 0 \bmod \mathcal P.
\]
\end{corollary}

The point of departure for this paper is a period formula for $L(2, f\otimes g\otimes h)$ due to Gross and Kudla. Beginning with the work of Harris and Kudla \cite{harris:1991} the central value of $L(s, f\otimes g\otimes h)$ has been linked to certain period integrals involving $f$, $g$ and $h$ or their Jacquet-Langlands transfers to multiplicative groups of quaternion algebras. In the case of squarefree level and weight two the work of Gross and Kudla \cite{gross:1992} made this link precise by providing an exact formula for $L(2, f\otimes g\otimes h)$ in terms of the period integral (which in this case is a finite sum) on the multiplicative group of a certain quaternion algebra. More precisely let $D$ denote the quaternion algebra over $\Q$ which is ramified at $N$ and $\infty$. We fix a maximal order $R$ in $D$ and let $S = \{ I_1, \hdots, I_n\}$ denote a (finite) set of representatives for the equivalence classes of left $R$-ideals in $D$. The Jacquet-Langlands correspondence assigns to each $f \in \F_2(N)$ a function $f'$ on $S$. The formula of Gross-Kudla (recalled in Theorem \ref{GK1} below) yields,
\[
4 \pi N \frac{L(2, f\otimes g\otimes h)}{(f, f)(g, g)(h, h)} = \left(\sum_{i=1}^n w_i^2 f'(I_i) g'(I_i) h'(I_i)\right)^2
\]
for any $f, g, h \in \F_2(N)$; here $w_i = \# R_i^\times/2$ where $R_i$ denotes the right order associated to the ideal $I_i$. The question of computing the average for the central value of the triple product $L$-function is thus turned into one about functions on the finite set $S$. We carry out this analysis in Section \ref{section:avgs}.

The results of this paper could also have been obtained via the relative trace formula. In a previous work \cite{me:average} we obtained exact formulas for averages of central values of twisted quadratic base change $L$-functions associated to Hilbert modular forms. In that paper we used an adelic relative trace formula together with a period formula due to Waldspurger \cite{waldspurger:1985}. In \cite[Section 1.2.3]{me:average} the relative trace formula approach was recast in more classical terms for the case of modular forms of weight two and prime level. This present paper can also be viewed as a classical version of a relative trace formula; in this case one would construct a relative trace formula by integrating the automorphic kernel for $D^\times\times D^\times$ against a fixed automorphic form on $D^\times$.

The restriction in this paper to the case of prime level and weight 2 is for simplicity, in particular we do not need to deal with oldforms. The identity of Gross and Kudla has been further generalized by B\"ocherer and Schulze-Pillot \cite{bocherer:1996} to more general levels and weights, Watson \cite{watson} and finally Ichino \cite{ichino:2008} who gives an essentially complete treatment. One could treat more general levels and weights (with certain restrictions on the weights of $f$, $g$ and $h$ relative to each other) using the period formula from \cite{bocherer:1996} however it would perhaps be preferable to use the adelic period formulas coming from \cite{ichino:2008} and work with an adelic relative trace formula. Furthermore in this way one could readily work over a general totally real number field and treat the case of triple product $L$-functions attached to Hilbert modular forms.

\section{Modular forms and $L$-functions}
\label{sec:modforms}

We fix a prime $N$. We let $M_2(N)$ denote the space of modular forms of level $N$ and weight 2 and let $S_2(N)$ denote the subspace of cuspforms. The Petersson inner product on $S_2(N)$ is normalized by,
\[
(f_1, f_2) = 8 \pi^2 \int_{\Gamma_0(N) \bs \calH} f_1(z) \overline{f_2(z)} \ {dx \ dy}.
\]
We let $\F_2(N)$ denote the set of normalized Hecke eigenforms in $S_2(N)$. The size of $\F_2(N)$ when $N$ is prime (see for example \cite[Theorem 1] {martin:2005}) is given by,
\begin{align}\label{dimn}
| \F_2(N) | = \left\{
  \begin{array}{ll}
    \frac{N-1}{12} - 1; & N \equiv 1 \bmod 12 \\
    \frac{N-1}{12} - \frac 13; & N \equiv 5 \bmod 12 \\
    \frac{N-1}{12} - \frac 12; & N \equiv 7 \bmod 12 \\
    \frac{N-1}{12} + \frac 16; &N \equiv 11 \bmod 12.
  \end{array}
\right.
\end{align}

We now recall from \cite{gross:1987} Eichler's work \cite{eichler:1955a}, \cite{eichler:1955b} on modular forms and quaternion algebras. Let $D$ denote the quaternion division algebra over $\Q$ which is ramified at $N$ and $\infty$. We fix a maximal order $R$ in $D$ and take $S = \{ I_1, \hdots, I_n\}$ to be a set of representatives for the equivalence classes of left $R$-ideals. To each ideal $I_i$ one associates the maximal right order
\[
R_i = \left\{ x \in D : I_i x \subset I_i \right\}.
\]
We set $w_i = \# R_i^\times/2$.

For later use we recall Eichler's mass formula \cite[(1.2)]{gross:1987},
\begin{eqnarray}
\label{eichlermass}
\sum_{i=1}^n \frac{1}{w_i} = \frac{N-1}{12}.
\end{eqnarray}
In Table \ref{table} below we recall from \cite[Table 1.3]{gross:1987} the values for $n$ and $\{w_i\}$ depending on $N$. We assume $N>3$.
\begin{center}
\begin{table}[h]
\caption{}
\begin{tabular}{|c|c|c|}
  \hline
  $N$ & $\{w_i\}$ & $n$\\
  \hline
  $\equiv 1 \bmod 12$ & $\{1, \hdots, 1\}$ & $\frac{N-1}{12}$ \\
  $\equiv 5 \bmod 12$ & $\{3, 1, \hdots, 1\}$ & $\frac{N+7}{12}$ \\
  $\equiv 7 \bmod 12$ & $\{2, 1, \hdots, 1\}$ & $\frac{N+5}{12}$ \\
  $\equiv 11 \bmod 12$ & $\{3, 2, 1, \hdots, 1\}$ & $\frac{N+13}{12}$ \\
  \hline
\end{tabular}
\label{table}
\end{table}
\end{center}
Let $M_2^D(N)$ denote the space of complex valued functions on $S$ with inner product defined by,
\[
\langle g_1, g_2\rangle = \sum_{i=1}^n w_i g_1(I_i) \overline{g_2(I_i)}.
\]
For each $i$ with $1 \leq i \leq n$ we set $e_i \in M_2^D(N)$ equal to the characteristic function of $I_i$. We note that,
\[
\langle e_i, e_j\rangle = \delta_{i, j} w_i.
\]
We also define,
\[
e= \sum_{i=1}^n \frac{1}{w_i} e_i \in M_2^D(N)
\]
and note that,
\begin{equation}\label{eisen}
\langle e, e\rangle = \sum_{i=1}^n \frac{1}{w_i} = \frac{N-1}{12},
\end{equation}
by \eqref{eichlermass}. Let $S_2^D(N) \subset M_2^D(N)$ denote the orthogonal complement of $e$ in $M_2^D(N)$, i.e. $S_2^D(N)$ consists of those
\[
\sum_{i=1}^n a_i e_i \in M_2^D(N)
\]
such that
\[
\sum_{i=1}^n a_i = 0.
\]
Let $\mathbf T^N$ denote the Hecke algebra away from $N$. Then there is a natural action of $\mathbf T^N$ on $S_2^D(N)$ as a family of self dual and self-adjoint operators; see \cite[Section 4]{gross:1987}. The Jacquet-Langlands correspondence, which in this special case is already proven in \cite{eichler:1955a} and \cite{eichler:1955b}, yields an isomorphism between $S_2(N)$ and $S_2^D(N)$ as modules over $\mathbf T^N$. Thus if
\[
f = \sum_{m=1}^\infty a_m q^m \in \F_2(N)
\]
then there exists a non-zero $f' \in S_2^D(N)$, which is well defined up to scaling by multiplicity one, such that
\[
T_m f' = a_m f'
\]
for all $T_m \in \mathbf T^N$. For each $f \in \F_2(N)$ we fix such an $f' \in S^D_2(N)$ normalized so that $\langle f', f'\rangle = 1$ and when we write
\[
f' = \sum_{i=1}^n \lambda_i(f) e_i
\]
each $\lambda_i(f)\in \R$. The existence of $f'$ follows from the self dual and self-adjoint properties of the Hecke algebra acting on $S_2^D(N)$; see \cite[Proposition 10.2]{gross:1992}.
We note that $f'$ is well defined up to multiplication by $\pm 1$.

We recall that by \cite[Proposition 4.4]{gross:1987}, for $m \geq 1$ and $i=1, 2, \dots, n$,
\[
T_m e_i = \sum_{j=1}^n B_{i j}(m)e_j,
\]
where $B(m) = (B_{ij}(m))$ is the Brandt matrix; see \cite[(1.5)]{gross:1987}. As a direct result of this and \cite[Proposition 2.7.1 and 2.7.6]{gross:1987},
\[
T_m e = \sigma(m)_N e
\]
where
\[
\sigma(m)_N = \sum_{d| m, (d, N)=1}d.
\]
By \cite[Proposition 1.9]{gross:1987}
\[
\tr(B(m))=\sum_{s^2\leq 4m} H_N(4m-s^2)
\]
where $H_N(D)$ is defined below.

Let $\mathcal O_{-D}$ be the order of discriminant $-D$, $h(d)$ be the class number of binary quadratic forms of discriminant $d$ and
\[
u(d)=\begin{cases}
3; & d=-3
\\
2; & d=-4
\\
1; & \text{ otherwise. }
\end{cases}
\]
Then we define
\[
H(D)=\sum_{df^2=-D} \frac{h(d)}{u(d)}
\] and finally
\begin{align*}
H_N(D)=\begin{cases}
0; & \text{$N$ splits in } \mathcal O_{-D}
\\
H(D); & \text{$N$ inert in $\mathcal O_{-D}$}
\\
\frac{1}{2} H(D); & \text{$N$ ramified in $\mathcal O_{-D}$ and $N$ does not divide the conductor of $\mathcal O_{-D}$}
\\
H_N(D/N^2); & N \text{ divides the conductor of $\mathcal O_{-D}$}
\\
\frac{N-1}{24}; & D=0.
\end{cases}
\end{align*}
By \cite{eichler:1955b},
\begin{equation}\label{trace}
 \tr(T_m |_{S_2(N)})+ \sigma(m)_N=\tr(B(m))=\sum_{s^2\leq 4m} H_N(4m-s^2).
\end{equation}

We take a normalized Hecke eigenform
\[
f = \sum_{m=1}^\infty a_m q^m \in \F_2(N).
\]
We recall one defines the $L$-function of $f$ by the Dirichlet series
\[
L(s, f) = \sum_{m=1}^\infty \frac{a_m}{m^s}.
\]
Let $\chi$ be a Dirichlet character of conductor $M$, then one defines,
\[
L(s, f \otimes \chi) = \sum_{m=1}^\infty \frac{a_m \chi(m)}{m^s}.
\]
As is well known these $L$-functions satisfy an analytic continuation to $\C$ and, with this normalization, a functional equation relating $s$ to $2-s$.

Suppose now $f, g, h \in S_2(N)$ and three (not necessarily distinct) normalized Hecke eigenforms. We write
\[
f = \sum_{m=1}^\infty a_m q^m, g = \sum_{m=1}^\infty b_m q^m, h = \sum_{m=1}^\infty c_m q^m.
\]
For each prime $p \neq N$ we write
\[
a_p = \alpha_{p, 1}  + \alpha_{p, 2}, b_p = \beta_{p, 1} + \beta_{p, 2}, c_p = \gamma_{p, 1} + \gamma_{p, 2}
\]
with
\[
\alpha_{p, 1} \alpha_{p, 2} = \beta_{p, 1} \beta_{p, 2} = \gamma_{p, 1} \gamma_{p, 2} = p.
\]
We also note that,
\[
a_N, b_N, c_N \in \{\pm 1\}.
\]
The triple product $L$-function is defined by an Euler product
\[
L(s, f \otimes g \otimes h) = \prod_p L_p(s, f\otimes g\otimes h),
\]
which converges for $\Re s > 5/2$, where for $p \neq N$,
\[
L_p(s, f\otimes g\otimes h) = \prod_{i=1}^2 \prod_{j=1}^2 \prod_{k=1}^2 \frac{1}{1 - \alpha_{p, i} \beta_{p, j} \gamma_{p, k} p^{-s}},
\]
and at $N$,
\[
L_N(s, f\otimes g\otimes h) = \frac{1}{1 - a_N b_N c_N N^{-s}} \frac{1}{(1 - a_N b_N c_N N^{1-s})^2}.
\]
We define,
\[
L_\infty(s, f\otimes g \otimes h) = (2 \pi)^{3 - 4s} \Gamma(s) \Gamma(s-1)^3
\]
and
\[
\Lambda(s, f\otimes g \otimes h) = L(s, f\otimes g\otimes h) L_\infty(s, f\otimes g\otimes h).
\]
Then (\cite[Proposition 1.1]{gross:1992} for this case) the function $\Lambda(s, f\otimes g\otimes h)$ has an analytic continuation to the whole complex plane and satisfies the functional equation
\[
\Lambda(s, f\otimes g \otimes h) = \eps_{f, g, h} N^{10 - 5s} \Lambda(4 - s,  f\otimes g\otimes h),
\]
where
\[
\eps_{f, g, h} = a_N b_N c_N.
\]

\section{Period formulas}

The main results of this paper are obtained from relations between central $L$-values and period integrals obtained in \cite{gross:1992} and \cite{gross:1987}. We now recall these results.

\begin{theorem} (\cite[Corollary 11.3]{gross:1992}) \label{GK1} Let $N$ be prime and let $f, g, h \in \F_2(N)$. Then,
\[
4 \pi N \frac{L(2, f\otimes g\otimes h)}{(f, f)(g, g)(h, h)} = \left(\sum_{i=1}^n w_i^2 \lambda_i(f) \lambda_i(g) \lambda_i(h)\right)^2.
\]
\end{theorem}

For a fundamental discriminant $-d < 0$, let $\chi_{-d}$ denote the unique primitive quadratic character of conductor $d$ such that $\chi_{-d}(-1)=-1$.

We shall also need the following special cases of \cite[Corollary 11.6]{gross:1987}.

\begin{theorem}\label{thm:gross4}
Let $N$ be a prime such that $N \equiv 3 \bmod 4$. Then there exists a unique $k$ with $1 \leq k \leq n$ such that $w_k = 2$ and
\[
2 \frac{L(1, f) L(1, f \otimes \chi_{-4})}{(f, f)} = \lambda_k(f)^2,
\]
for any $f \in \F_2(N)$.
\end{theorem}

\begin{theorem}\label{thm:gross3}
Let $N$ be a prime such that $N \equiv 2 \bmod 3$. Then there exists a unique $k$ with $1 \leq k \leq n$ such that $w_k = 3$ and
\[
\sqrt{3} \frac{L(1, f) L(1, f \otimes \chi_{-3})}{(f, f)} = \lambda_k(f)^2,
\]
for any $f \in \F_2(N)$.
\end{theorem}

%To see that \cite[Corollary 11.6]{gross:1987} reduces to the two theorems above in these special cases, we first note that the class number of both $\Q(\sqrt{-1})$ and $\Q(\sqrt{-3})$  is one, and that the number of units in the ring of integers of $\Q(\sqrt{-1})$ is $4$ and of $\Q(\sqrt{-3})$ is $6$. Finally, we observe that $e_{k,f}$, the projection of $e_k$ to $f'$, is $\lambda_k(f)w_kf'$ so $\langle e_{k,f}, e_{k,f}\rangle=\lambda_k(f)^2w_k^2$.

\section{Averages of central $L$-values}\label{section:avgs}

In this section we apply the period formula of Gross and Kudla to the study of
\[
\frac{L(2, f\otimes g\otimes h)}{(f, f) (g, g) (h, h)}
\]
as one varies over $f, g, h \in \F_2(N)$. We begin with the average over one form.

\begin{lemma}
For $N$ prime and $g, h \in \F_2(N)$,
\[
{4 \pi N}\sum_{f \in \F_2(N)} \frac{L(2, f\otimes g\otimes h)}{(f, f)(g,g)(h,h)} = \sum_{i=1}^n w_i^3 \lambda_i(g)^2 \lambda_i(h)^2 - \frac{12}{N-1} \delta_{g, h},
\]
where $\delta_{g,h}$ equals $1$ if $g=h$ and $0$ otherwise.
\end{lemma}

\begin{proof}
By Theorem \ref{GK1} (Corollary 11.3 in \cite{gross:1992}),
\[
{4\pi N}\sum_{f \in \F_2(N)} \frac{L(2, f\otimes g\otimes h)}{(f,f)(g,g)(h,h)}
=\sum_{f \in \F_2(N) }  \left(\sum_{i=1}^n w_i^2 \lambda_i(f) \lambda_i(g) \lambda_i(h)\right)^2.
\]
As
\[
\sum_{i=1}^n w_i^2 \lambda_i(f) \lambda_i(g) \lambda_i(h) =\left\langle f', \sum_{i=1}^n w_i {\lambda_i(g)\lambda_i(h)}e_i \right\rangle
\]
and $\{f' : f \in \F_2(N) \} \cup \{e/\sqrt{\langle e, e \rangle} \}$ is an orthonormal basis for $M_2^D(N)$, we have
\begin{eqnarray*}
\sum_{f \in \F_2(N) } \left(\sum_{i=1}^n w_i^2 \lambda_i(f) \lambda_i(g)\lambda_i(h)\right)^2+  \frac{1}{\langle e, e\rangle}{ \left(\sum_{i=1}^n  w_i \lambda_i(g)\lambda_i(h)\right)^2}
=\sum_{i=1}^n {w_i^3  \lambda_i(g)^2\lambda_i(h)^2},
\end{eqnarray*}
by Parseval's identity. The lemma now follows from \eqref{eisen} and by observing that
\[
\sum_{i=1}^n  w_i \lambda_i(g)\lambda_i(h)= \langle g', h' \rangle=\delta_{g,h}.
 \]
\end{proof}

We now sum the previous formula over $g$ weighted against a Hecke eigenvalue $a_m(g)$ to obtain the following theorem.

\begin{theorem}\label{thm:avg}
Let $N$ be prime with $N = 11$ or $N > 13$. Then for any $h \in \F_2(N)$,
\begin{align*}
4 \pi N \sum_{f, g \in \F_2(N)} \frac{L(2, f\otimes g\otimes h)}{(f, f)(g, g)(h, h)} a_m(g)
\end{align*}
equals
\[
\sum_{i=1}^n w_i^2 B_{i i}(m)  \lambda_i(h)^2 - \frac{12 \sigma(m)_N}{N-1}  - \frac{12}{N - 1} a_m(h).
\]
\end{theorem}

\begin{proof}
By the previous lemma,
\begin{equation}
4 \pi N\sum_{f, g \in \F_2(N)} \frac{L(2, f\otimes g\otimes h)}{(f, f)(g, g) (h,h)}  a_m(g) = \sum_{g \in \F_2(N)} a_m(g)\sum_{i=1}^n w_i^3 \lambda_i(g)^2 \lambda_i(h)^2 - \sum_{g\in \F_2(N)}\frac{12 }{N-1} a_m(g)\delta_{g, h}.
\label{eqn:sum}
\end{equation}
Clearly,
\[
\sum_{g\in \F_2(N)}\frac{12 }{N-1} a_m(g)\delta_{g, h} = \frac{12}{N-1} a_m(h),
\]
provided $\F_2(N)$ is nonempty. We have,
\begin{eqnarray*}
\sum_{g \in \F_2(N)} a_m(g)\sum_{i=1}^n  w_i ^3  \lambda_i(g)^2 \lambda_i(h)^2
&=& \sum_{i=1}^n w_i^3  \lambda_i(h)^2 \sum_{g \in \F_2(N)} a_m(g) \lambda_i(g)^2
\\
&=& \sum_{i=1}^n  w_i  \lambda_i(h)^2  \sum_{g \in \F_2(N)} \langle T_m g', e_i \rangle \langle g', e_i \rangle.
\end{eqnarray*}
Recalling that $\{ g' : g \in \F_2(N)\} \cup \{ e/\sqrt{\langle e, e\rangle}\}$ is an orthonormal basis of $M_2^D(N)$ and $T_m$ is a self-adjoint operator we obtain,
\begin{align*}
\sum_{g \in \F_2(N)} \langle T_m g', e_i \rangle \langle g', e_i \rangle + \frac{\langle T_m e, e_i \rangle \langle e , e_i \rangle}{\langle e, e \rangle}= \langle T_m e_i, e_i \rangle.
\end{align*}
Hence
\begin{align*}
\sum_{g \in \F_2(N)} \langle T_m g', e_i \rangle \langle g', e_i \rangle= w_i B_{i i}(m) - \frac{12 \sigma(m)_N}{N-1}.
\end{align*}
Thus expression \eqref{eqn:sum} is equal to the sum of
\[
\sum_{i=1}^n w_i^2 B_{i i}(m) \lambda_i(h)^2
\]
and
\[
 - \frac{12 \sigma(m)_N}{N-1}  \sum_{i=1}^n w_i \lambda_i(h)^2= - \frac{12 \sigma(m)_N}{N-1}.
\]
\end{proof}

We now specialize this theorem to $m=1$ to get a more explicit formula in this case.

\begin{corollary}\label{cor:avg}
Let $N$ be prime with $N=11$ or $N>13$. Then for any $h \in \F_2(N)$,
\begin{align*}
4 \pi N \sum_{f, g \in \F_2(N)} \frac{L(2, f\otimes g\otimes h)}{(f, f)(g, g)}
\end{align*}
equals
\begin{align*} \left(1-\frac{24}{N-1}\right) (h,h)+
\begin{cases}
0; & N \equiv 1 \bmod 12\\
6\sqrt 3{L(1, h) L(1, h\otimes \chi_{-3})};
& N \equiv 5 \bmod 12\\
4{L(1, h) L(1, h\otimes \chi_{-4})}; & N \equiv 7 \bmod 12\\
6\sqrt 3{L(1, h) L(1, h \otimes \chi_{-3})}+4{L(1, h) L(1, h \otimes \chi_{-4})}; & N \equiv 11 \bmod 12.
\end{cases}
\end{align*}
\end{corollary}

\begin{proof}
Setting $m=1$ in Theorem \ref{thm:avg} gives,
\begin{eqnarray*}
4 \pi N\sum_{f, g \in \F_2(N)} \frac{L(2, f\otimes g\otimes h)}{(f, f)(g, g) (h,h)} &=&\sum_{i=1}^n  w_i^2 \lambda_i(h)^2  - \frac{24}{N-1}.
\end{eqnarray*}
Now we note that
\begin{align*}
\sum_{i=1}^n  w_i^2 \lambda_i(h)^2  = \sum_{i=1}^n  (w_i^2-w_i) \lambda_i(h)^2 +\sum_{i=1}^n  w_i \lambda_i(h)^2 =\sum_{i=1}^n  (w_i^2-w_i) \lambda_i(h)^2 +1.
\end{align*}
Thus,
\begin{align}\label{eqn1}
4 \pi N \sum_{f, g \in \F_2(N)} \frac{L(2, f\otimes g\otimes h)}{(f, f)(g, g)(h,h)} = 1-\frac{24}{N-1}+\sum_{i=1}^n  (w_i^2-w_i) \lambda_i(h)^2.
\end{align}
For the final term we note that the only terms which contribute are those for which $w_i > 1$. From Table \ref{table} the only possibilities are $w_i \in \{2, 3\}$ and in these cases we can interpret $\lambda_i(h)^2$ as a special $L$-value associated to $h$ by Theorems \ref{thm:gross4} and \ref{thm:gross3}. Finally multiplying both sides of the identity by $(h, h)$ yields the corollary.
\end{proof}

We recall the well known fact that $|\F_2(N)|\sim \phi(N)$ along with the bound
\[
(f,f) \ll N(\log N)^2,
\]
which follows from the Ramanujan conjecture proven by Deligne. These facts together with Theorem \ref{thm:avg} imply that,
\[
\frac{1}{|\F_2(N)|^2} \sum_{f, g\in \F_2(N)} L(2, f\otimes g \otimes h) \ll_\epsilon N^\epsilon,
\]
which agrees with the Lindel\"{o}f conjecture on the average.

Finally, we sum over all three forms against one Hecke eigenvalue.
Let
\[
R_{-d}(m)=|\{ \mathfrak a \subset \mathcal O_{\mathbf Q(\sqrt{-d})} : \Nm(\mathfrak a)=m\}|.
\]

\begin{proposition}\label{prop:avg}
Let $N$ be prime with $N=11$ or $N>13$. Then
\begin{align*}
4 \pi N \sum_{f, g, h \in \F_2(N)} \frac{L(2, f\otimes g\otimes h)}{(f, f)(g, g)(h,h)} a_m(h)
\end{align*}
equals
\begin{align*}
&\left(1-\frac{24}{N-1}\right)  \left( \sum_{s^2 \leq 4m} H_N(4m - s^2) -\sigma(m)_N\right)
\\
&+\begin{cases}
0; & N \equiv 1 \bmod 12\\
2 R_{-3}(m) - \frac{8}{N-1} \sigma(m)_N; & N \equiv 5 \bmod 12\\
R_{-4}(m) - \frac{6}{N-1} \sigma(m)_N; & N \equiv 7 \bmod 12\\
2 R_{-3}(m) + R_{-4}(m) - \frac{14}{N-1} \sigma(m)_N; & N \equiv 11 \bmod 12.
\end{cases}
\end{align*}

\end{proposition}
\begin{proof}
By Corollary \ref{cor:avg} this equals
\begin{align*}
&
 \sum_{h \in \F_2(N)}\left(1-\frac{24}{N-1}\right) a_m(h)
\\ +& \sum_{h \in \F_2(N)} \frac{a_m(h)} {(h,h)} \times
\begin{cases}
0; & N \equiv 1 \bmod 12\\
6\sqrt 3{L(1, h) L(1, h\otimes \chi_{-3})};
& N \equiv 5 \bmod 12\\
4{L(1, h) L(1, h\otimes \chi_{-4})}; & N \equiv 7 \bmod 12\\
6\sqrt 3{L(1, h) L(1, h \otimes \chi_{-3})}+4{L(1, h) L(1, h \otimes \chi_{-4})}; & N \equiv 11 \bmod 12.
\end{cases}
\end{align*}
Which equals
\begin{align*}& \left(1-\frac{24}{N-1}\right) \tr(T_m| S_2(N))
\\
+&
\begin{cases}
0; & N \equiv 1 \bmod 12\\
6\sqrt 3   \sum_{h \in \F_2(N)} \frac{L(1, h)L(1, h\otimes \chi_{-3})} {(h,h)}a_m(h);
& N \equiv 5 \bmod 12\\
4  \sum_{h \in \F_2(N)} \frac{L(1, h)L(1, h\otimes \chi_{-4})} {(h,h)}a_m(h); & N \equiv 7 \bmod 12\\
6\sqrt 3  \sum_{h \in \F_2(N)} \frac{L(1, h)L(1, h\otimes \chi_{-3})} {(h,h)}a_m(h)+4  \sum_{h \in \F_2(N)} \frac{L(1, h)L(1, h\otimes \chi_{-4})} {(h,h)}a_m(h); & N \equiv 11 \bmod 12.
\end{cases}
\end{align*}
We now recall the following average value formulas which follow from \cite{ramakrishnan:exact} where we note that we have adjusted the formula so that $(h,h)$ is normalized as in this paper:
\begin{align*}
6\sqrt 3\sum_{h\in \F_2(N)} \frac{{L(1, h) L(1, h\otimes \chi_{-3})}}{(h,h)}a_m(h)=\frac{2}{3}\left( 3 R_{-3}(m)-\frac{12}{N-1} \sigma(m)_N\right)
\end{align*}
and
\begin{align*}
4\sum_{h\in \F_2(N)} \frac{{L(1, h) L(1, h\otimes \chi_{-4})}}{(h,h)}a_m(h)=\frac{1}{2}\left( 2R_{-4}(m)-\frac{12}{N-1} \sigma(m)_N\right).
\end{align*}
The result now follows by applying \eqref{trace}.
\end{proof}

By \cite[Corollary  11.2(b)]{gross:1992} for $f, g, h\in \F_2(N)$ and any $\sigma \in \Aut(\C)$
\[
\sigma \left( 4\pi N \frac{L(2, f\otimes g\otimes h)}{(f, f)(g, g)(h, h)} \right) = 4\pi N \frac{L(2, f^\sigma \otimes g^\sigma \otimes h^\sigma)}{(f^\sigma, f^\sigma)(g^\sigma, g^\sigma)(h^\sigma, h^\sigma)},
\]
where $f^\sigma$, $g^\sigma$ and $h^\sigma$ denote the modular forms obtained by applying $\sigma$ to the Fourier coefficients of $f$, $g$ and $h$. Since $f^\sigma, g^\sigma, h^\sigma \in \F_2(N)$ we see that,
\[
4\pi N\sum_{f, g, h \in \F_2(N)} \frac{L(2, f\otimes g\otimes h)}{(f, f)(g, g)(h, h)}
\]
is fixed by every automorphism of $\C$ and hence is rational. By setting $m=1$ in Proposition \ref{prop:avg}  we can compute this rational number.

\begin{corollary}\label{cor:avg3}
For $N$ prime,
\begin{align}\label{avg3}
4\pi N\sum_{f, g, h \in \F_2(N)} \frac{L(2, f\otimes g\otimes h)}{(f, f)(g, g)(h, h)}
= \begin{cases}
 \frac{N-25}{12} ; & N \equiv 1 \bmod 12\\
\frac{N-5}{12};
& N \equiv 5 \bmod 12\\
\frac{(N-7)(N-13)}{12(N-1)}; & N \equiv 7 \bmod 12\\
\frac{N^2+12N-229}{12(N-1)}
; & N \equiv 11 \bmod 12.
\end{cases}
\end{align}
\end{corollary}

Setting $m = N$ and for simplicity restricting to $N \equiv 1 \bmod 12$ we see that,

\begin{corollary}
For $N$ prime and $N\equiv 1 \bmod 12$,
\begin{align*}
4\pi N\sum_{f, g, h \in \F_2(N)} \frac{L(2, f\otimes g\otimes h)}{(f, f)(g, g)(h, h)}a_N(h)
=\left(1-\frac{24}{N-1}\right) \left( \frac 12 h(-4N)-1\right).
\end{align*}

\end{corollary}

\section{Consequences of the average value formulas}
\label{sec:conseq}

The exact formulas in Section \ref{section:avgs} can be used to obtain information on the non-vanishing of $L(2, f\otimes g \otimes h)$. Our first result is a direct consequence of Corollary \ref{cor:avg} and the non-negativity of $L(1,h)L(1, h\otimes \chi_{-d})$.

\begin{corollary}
Let $N>25$ be prime. For each $h \in \F_2(N)$ there exist $f, g \in \F_2(N)$ such that $L(2, f\otimes g\otimes h ) \neq 0$.
\end{corollary}

Using the convexity bound for $L(2, f\otimes g\otimes h)$ together with Corollary \ref{cor:avg} one obtains

\begin{corollary}
Let $N>25$ be prime. For $h \in \F_2(N)$,
\[
\#\{ (f,g)\in \F_2(N)\times \F_2(N): L(2, f\otimes g \otimes h)\neq 0\} \gg_\epsilon N^{3/4-\epsilon}.
\]
\end{corollary}

\begin{proof}
From  \cite{hoffstein:1994},
\[
\frac{1}{(f,f)}\ll \frac{(\log N)^2}{N}.
\]
Applying this together with the non-negativity of $L(1,h)L(1, h\otimes \chi_{-d})$ to Corollary \ref{cor:avg} we have
\[
\sum_{f, g\in \F_2(N)} L(2, f\otimes g \otimes h) \gg N^2(\log N)^{-6}.
\]
The corollary now follows from the convexity bound $L(2, f\otimes g \otimes h) \ll_{\epsilon} N^{5/4 + \epsilon}$ \cite[(5.21)]{iwaniec:2004}.
\end{proof}

We now define
\[
L^{alg}(2, f\otimes g \otimes h)= 4\pi N \frac{L(2, f \otimes g\otimes h)}{(f,f)(g,g)(h,h)}.
\]
By \cite[Corollary 11.2(b)]{gross:1992}, $L^{alg}(2, f\otimes g \otimes h)$ lies in the subfield of $\C$ generated by the Fourier coefficients of $f$, $g$ and $h$ and hence is algebraic.

\begin{corollary}
Let $p$ be prime and let $\mathcal P$ be a place in $\overline \Q$ above $p$. Let $N$ be prime such that $N \equiv 1 \bmod 12$ and $p\nmid(N-25)$. Then for any $h \in \F_2(N)$, there exist $f, g \in \F_2(N)$ such that
\[
L^{alg}(2, f\otimes g \otimes h) \not\equiv 0 \bmod \mathcal P.
\]
\end{corollary}

\begin{proof}
We note from Corollary \ref{cor:avg} that
\begin{align*}
4\pi N\sum_{f, g  \in \F_2(N)} \frac{L(2, f\otimes g\otimes h)}{(f, f)(g, g)(h, h)}
=
\frac{N-25}{N-1}.
\end{align*}
The corollary is now immediate.
\end{proof}

\section{Numerical verification}

In this section we check our formulas with some numerical calculations.  We note that when $N=11, 17$ or $19$, $|\F_2(N)|=1$. Thus the left hand side of \eqref{avg3} in Corollary \ref{cor:avg3}  has only one term. The following values can be deduced from  \cite[Table 12.5]{gross:1992} and the period formula (Theorem \ref{GK1}  which is Corollary 11.3 in \cite{gross:1992}),
\begin{align}\label{numerics}
4\pi N\sum_{f, g, h \in \F_2(N)} \frac{L(2, f\otimes g\otimes h)}{(f, f)(g, g)(h, h)} =
\begin{cases}
\frac{1}{5}; & N=11
\\
{1}; & N=17
\\
\frac{1}{3}; & N=19.
\end{cases}
\end{align}
These values agree with Corollary \ref{cor:avg3}.

We now consider the case that $N = 37$. In this case $|\F_2(37)| = 2$, $n = 3$ and $w_i = 1$ for each $i$ with $1 \leq i \leq 3$. Furthermore if we enumerate the set $S = \{ I_1, I_2, I_3\}$ as in \cite[Table 12.5]{gross:1992} then there exists $f_1 \in \F_2(37)$ such that,
\[
f_1' = \frac{1}{\sqrt 6} ( 2 e_1 - e_2 - e_3).
\]
We also have,
\[
e = e_1 + e_2 + e_3.
\]
If we write $\F_2(37) = \{f_1, f_2\}$ then $f_2'$ is a unit vector in $M_2^D(37)$ which is orthogonal to $f_1'$ and $e$ and hence can be taken to be
\[
f_2' = \frac{1}{\sqrt 2}(e_2 - e_3).
\]
We now use Theorem \ref{GK1}  (Corollary 11.3 in \cite{gross:1992}) to compute the relevant triple product $L$-functions. We have,
\begin{align*}
4\cdot  37 \pi \frac{L(2, f_1\otimes f_1\otimes f_1)}{(f_1, f_1)(f_1, f_1)(f_1, f_1)} &= \frac 16\\
4 \cdot 37 \pi \frac{L(2, f_1\otimes f_1\otimes f_2)}{(f_1, f_1)(f_1, f_1)(f_2, f_2)} &= 0\\
4 \cdot 37 \pi \frac{L(2, f_1\otimes f_2\otimes f_2)}{(f_1, f_1)(f_2, f_2)(f_2, f_2)} &= \frac 16\\
4 \cdot 37 \pi \frac{L(2, f_2\otimes f_2\otimes f_2)}{(f_2, f_2)(f_2, f_2)(f_2, f_2)} &= 0.
\end{align*}
We have
\[
4 \cdot 37 \pi\sum_{f, g \in \F_2(37)} \frac{L(2, f\otimes g\otimes f_1)}{(f, f)(g, g)(f_1, f_1)} = \frac 16 + 2 \cdot 0 + \frac 16 + 0 = \frac 13
\]
and
\[
4 \cdot 37 \pi\sum_{f, g \in \F_2(37)} \frac{L(2, f\otimes g\otimes f_2)}{(f, f)(g, g)(f_2, f_2)} = 0 + 2 \cdot\frac 16 + 0 = \frac 13.
\]
Hence,
\[
4 \cdot 37 \pi\sum_{f, g, h \in \F_2(37)} \frac{L(2, f\otimes g\otimes h)}{(f, f)(g, g)(h, h)} = \frac 16 + 3 \cdot 0 + 3 \cdot\frac 16 + 0 = \frac 23.
\]
As one can readily check, these values agree with Corollaries \ref{cor:avg} and \ref{cor:avg3}.

\section{Acknowledgements}
This work was completed at the Centre de recherches math\'{e}matiques summer school  ``Automorphic Forms and $L$-Functions: Computational Aspects'' in June 2009 organized  by H. Darmon, E. Goren, M. Rubinstein and A. Str\"{o}mbergsson, and the authors would like to thank the organizers and the CRM for excellent working conditions. The authors thank the referee for helpful comments.

The first author was supported by  the Natural Sciences and Engineering Research Council of Canada. The second author was supported by National Science Foundation grant DMS-0758197.

\providecommand{\bysame}{\leavevmode\hbox to3em{\hrulefill}\thinspace}
\providecommand{\MR}{\relax\ifhmode\unskip\space\fi MR }
\providecommand{\MRhref}[2]{%
  \href{http://www.ams.org/mathscinet-getitem?mr=#1}{#2}
}
\providecommand{\href}[2]{#2}

\end{document}